\documentclass[twoside]{article}  

\usepackage[margin=2cm,bottom=2.5cm]{geometry}
\usepackage{amsmath,amsthm,amssymb,latexsym}
\usepackage[v2,tips]{xy}

\newcommand{\ap}{{\operatorname{AP}}}
\newcommand{\p}{{\operatorname{P}}}
\newcommand{\wap}{{\operatorname{WAP}}}

\newcommand{\aone}{\Box}
\newcommand{\atwo}{\Diamond}
\newcommand{\proten}{{\widehat{\otimes}}}

\newcommand{\mc}[1]{\mathcal{#1}}
\newcommand{\ip}[2]{{\langle {#1} , {#2} \rangle}}
\newcommand{\id}{\operatorname{id}}
\newcommand{\inten}{{\check{\otimes}}}

\theoremstyle{plain}
\newtheorem{proposition}{Proposition}[section]
\newtheorem{theorem}[proposition]{Theorem}
\newtheorem{corollary}[proposition]{Corollary}
\newtheorem{lemma}[proposition]{Lemma}
\theoremstyle{definition}

\begin{document}

\large
\title{Characterising weakly almost periodic functionals on the measure algebra}
\author{Matthew Daws}
\maketitle

\begin{abstract}
Let $G$ be a locally compact group, and consider the weakly-almost periodic
functionals on $M(G)$, the measure algebra of $G$, denoted by $\wap(M(G))$.
This is a C$^*$-subalgebra of the commutative C$^*$-algebra $M(G)^*$, and so has
character space, say $K_\wap$.  In this paper, we investigate properties of $K_\wap$.
We present a short proof that $K_\wap$ can naturally be turned into a
semigroup whose product is separately continuous: at the Banach algebra level, this product
is simply the natural one induced by the Arens products.  This is in complete agreement
with the classical situation when $G$ is discrete.  A study of how $K_\wap$ is related
to $G$ is made, and it is shown that $K_\wap$ is related to the weakly-almost periodic
compactification of the discretisation of $G$.  Similar results are shown for
the space of almost periodic functionals on $M(G)$.


\end{abstract}

\section{Introduction}

In \cite{daws2}, we developed a theory of corepresentations on reflexive Banach
spaces, and used this to show, in particular, that the space of weakly almost periodic
functionals on the measure algebra $M(G)$ forms a C$^*$-subalgebra of $M(G)^* =
C_0(G)^{**}$.  We write $\wap(M(G))$ for this space, so we see that $\wap(M(G)) = C(K_\wap)$
for some compact Hausdorff space $K_\wap$.  By analogy with the discrete case, when $M(G)=\ell^1(G)$
and when $\wap(M(G))$ can be identified with $\wap(G)$, we would expect $K_\wap$ to become
a semigroup in a natural way, perhaps with continuity properties, and perhaps with
some sort of universal property related to $G$.  For more information on weakly almost
periodic functionals, see \cite{ramirez} or \cite{BJM}; a recent study of the measure
algebra is \cite{DLS}.

In this paper, we show that $K_\wap$ does indeed carry a natural semigroup product which
is separately continuous.  By ``natural'', we mean that the product is directly induced by the
product on $G$: at the level of Banach algebras, this is simply the Arens product.
We show that, formally, the passage from $G$ to $K_\wap$ is a functor
between natural categories.  We might hope, as in the discrete case, to be able
to give a satisfactory description of $K_\wap$ in terms of $G$, but without reference to
specific algebras of functions.  We show some general properties of $K_\wap$, but at present
we fail to find such a description of $K_\wap$ purely in terms of $G$.

As well as weakly almost periodic functionals, one can consider almost periodic functionals.
We do this in the first section, which the reader my skip on a first reading,
if they are interested mainly in the weakly almost periodic case.
Here the functional analytic tools required are simpler, but this easier setting allows
us to develop some methods without undue worry about technicalities.  We also make links
with some classical notions, in particular, Taylor's Structure Semigroup for $M(G)$.

In the next section, we present a short proof that the character space of the
C$^*$-algebra of weakly almost periodic functionals becomes a semigroup whose product
is separately continuous.  The key idea is to use a suitable space of separately
continuous functions.  This was communicated to us be the anonymous referee of a previous
version of this paper.  The motivation of the construction in \cite{daws2} was to build
a theory which might be applicable to the non-commutative case; recently Runde has shown
in \cite{rundepp} that this is rather unlikely to work.  Similarly, the methods of the present
paper, being essentially the study of function spaces, also seems unlikely to generalise
directly to the non-commutative case.

So, we have that $\wap(M(G)) = C(K_\wap)$ where $K_\wap$ becomes a semigroup whose product
is separately continuous, a \emph{semitopological} compact semigroup.
In the final section we study $K_\wap$ as a semigroup, much in the theme of
Section~\ref{ap_sec}.

\medskip
\noindent\textbf{Acknowledgments:} I would like to thank Pekka Salmi for useful conversations,
especially about Taylor's Structure Semigroup, and for suggesting to look at \cite{dgh}.
I am indebted to the anonymous referee of an earlier version of paper who
suggested the idea, and some of the proofs, which forms Section~\ref{wap_case}.
Garth Dales, Tony Lau and Dona Strauss for comments on the exposition of this article.

\section{Commutative Hopf von Neumann algebras}\label{sec:intro}

We now quickly recall some notions and results from \cite{daws2}.
A commutative Hopf von Neumann algebra is a commutative von Neumann algebra
$L^\infty(X)$ equipped with a normal $*$-homomorphism $\Delta:L^\infty(X)
\rightarrow L^\infty(X\times X)$ which is coassociative in the sense that
$(\id\otimes\Delta)\Delta = (\Delta\otimes\id)\Delta$.  In classical situations
$\Delta$ is induced by some map $X\times X\rightarrow X$, but it is important for
our applications that we work with more generality.  The pre-adjoint
of $\Delta$, denoted by $\Delta_*$, induces a map $L^1(X)\proten L^1(X)
= L^1(X\times X) \rightarrow L^1(X)$ which is associative, turning $L^1(X)$
into a Banach algebra.  Here $\proten$ denotes the projective tensor product.
For the basics on tensor products, see \cite[Appendix~A.3]{dales},
\cite[Section~1.10]{palmer}
or \cite[Section~2, Chapter~IV]{tak}.

Our main reason for considering such objects is because, for suitable
$X$ and $\Delta$, we have that $M(G) = L^1(X)$.  Let us quickly recall how
to do this.  Define $\Phi : C_0(G)\rightarrow C(G\times G)$ by
\[ \Phi(f)(s,t) = f(st) \qquad (f\in C_0(G), s,t\in G). \]
Here $C(G\times G)$ is the space of bounded continuous functions on $G\times G$.
We can identify $C(G\times G)$ as a subspace of the dual of $M(G) \proten M(G)$ by
integration,
\[ \ip{F}{\mu\otimes\lambda} = \int_{G\times G} F(s,t) \ d\mu(s) \ d\lambda(t)
\qquad (F\in C(G\times G), s,t\in G). \]
Then, as $C_0(G)^{**} = M(G)^*$ is a commutative von Neumann algebra, there exists
a measure space $X$ with $C_0(G)^{**} = L^\infty(X)$, and so $M(G) = L^1(X)$
as a Banach space.  Thus we regard $\Phi$ as a $*$-homomorphism $C_0(G)
\rightarrow L^\infty(X\times X)$.
There exists a unique normal coassociative $*$-homomorphism
$\Delta:L^\infty(X)\rightarrow L^\infty(X\times X)$ such that
$\Delta \kappa_{C_0(G)}(f) = \Phi(f)$ for $f\in C_0(G)$.  Here, for a Banach space
$E$, $\kappa_E:E\rightarrow E^{**}$ is the canonical map from $E$ to its bidual.
A check shows that the pre-adjoint $\Delta_*$ induces the usual convolution
product on $M(G)$.  See \cite[Section~2.1]{daws2} for further details.

As in \cite{daws2}, it is convenient to work with the abstraction of
commutative Hopf von Neumann algebras, using $M(G)$ as our main example.

Let $\mc A$ be a Banach algebra.  We turn $\mc A^*$ into an $\mc A$-bimodule
in the usual way
\[ \ip{a\cdot\mu}{b} = \ip{\mu}{ba}, \quad
\ip{\mu\cdot a}{b} = \ip{\mu}{ab} \qquad (a,b\in\mc A, \mu\in\mc A^*). \]
We define $\mu\in\mc A^*$ to be \emph{weakly almost periodic} if the map
\[ R_\mu:\mc A\rightarrow \mc A^*, \quad a\mapsto a\cdot\mu \qquad (a\in\mc A), \]
is weakly compact.  We write $\mu\in\wap(\mc A)$.  Similarly, if $R_\mu$ is compact,
then $\mu$ is \emph{almost periodic}, written $\mu\in\ap(\mc A)$.  It is easy
to see that $\ap(\mc A)$ and $\wap(\mc A)$ a closed submodules of $\mc A^*$.
Here we used actions on the left, but we get the same concepts if we instead
look at the map $L_\mu(a) = \mu\cdot a$ for $a\in\mc A$.

We shall use the Arens products in a few places, so we define them here.  We define
contractive bilinear maps $\mc A^{**}\times\mc A^*,\mc A^*\times\mc A^{**}
\rightarrow\mc A^*$ by
\[ \ip{\Phi\cdot\mu}{a} = \ip{\Phi}{\mu\cdot a}, \quad
\ip{\mu\cdot\Phi}{a} = \ip{\Phi}{a\cdot\mu}
\qquad (a\in\mc A, \mu\in\mc A^*, \Phi\in\mc A^{**}). \]
Then we define contractive bilinear maps $\aone,\atwo:\mc A^{**}\times\mc A^{**}
\rightarrow \mc A^{**}$ by
\[ \ip{\Phi\aone\Psi}{\mu} = \ip{\Phi}{\Psi\cdot\mu}, \quad
\ip{\Phi\atwo\Psi}{\mu} = \ip{\Psi}{\mu\cdot\Phi}
\qquad (\mu\in\mc A^*,\Phi,\Psi\in\mc A^{**}). \]
These are actually associative algebra products such that $a\cdot\Phi =
\kappa_{\mc A}(a)\aone\Phi = \kappa_{\mc A}(a)\atwo\Phi$ for $a\in\mc A,
\Phi\in\mc A^{**}$, and similarly for $\Phi\cdot a$.  See \cite[Section~1.4]{palmer}
or \cite[Theorem~2.6.15]{dales} for further details.

\section{Almost periodic case}\label{ap_sec}

In this section we shall investigate further properties of $\ap(L^1(X))$
for a commutative Hopf von Neumann algebras $(L^\infty(X),\Delta)$.
This case is easier than the weakly almost periodic case,
and will allow us to build some general theory without added complication.
By \cite[Theorem~1]{daws2}, we know that $\ap(L^1(X))$ is a
C$^*$-subalgebra of $L^\infty(X)$, and so $\ap(L^1(X)) = C(K_\ap)$ for some
compact Hausdorff space $K_\ap$.

In the following proof, we write $\inten$ for the \emph{injective tensor product},
which for commutative C$^*$-algebras agrees with the \emph{minimal} or
\emph{spacial tensor product}; see, for example, \cite[Section~4, Chapter~IV]{tak}.

\begin{theorem}\label{get_top_semi}
Let $L^\infty(X)$ be a commutative Hopf von Neumann algebra, and let
$\ap(L^1(X))=C(K_\ap)$.  Then $\Delta$ restricts to a map $C(K_\ap)\rightarrow
C(K_\ap \times K_\ap)$, and hence naturally induces a jointly continuous
semigroup product on $K_\ap$.
\end{theorem}
\begin{proof}
As in the proof of \cite[Theorem~1]{daws2}, we know that $F\in\ap(L^1(X))$
if and only if $\Delta(F)\in L^\infty(X) \inten L^\infty(X)$.  That $\ap(L^1(X))$
is an $L^1(X)$-submodule of $L^\infty(X)$ is equivalent to
\[ (a\otimes\id)\Delta(F), (\id\otimes a)\Delta(F)\in \ap(L^1(X))
\qquad (a\in L^1(X), F\in\ap(L^1(X)) ). \]
As $L^\infty(X)$ is commutative, this is equivalent to $\Delta(F)
\in \ap(L^1(X)) \inten \ap(L^1(X))$ for $F\in\ap(L^1(X))$.  Thus $\Delta$ restricts
to give a $*$-homomorphism $C(K_\ap) \rightarrow C(K_\ap)\inten C(K_\ap) =
C(K_\ap\times K_\ap)$.  Hence there is a continuous homomorphism
$K_\ap\times K_\ap\rightarrow K_\ap$, which we shall write
as $(s,t)\mapsto st$, such that
\[ \Delta(f)(s,t) = f(st) \qquad (f\in C(K_\ap), s,t\in K_\ap). \]
As $\Delta$ is coassociative, it easily follows that this product on $K_\ap$
is associative, as required.
\end{proof}

It is almost immediate that $K_\ap$ can be characterised, rather abstractly, as follows.
Let $S$ be a compact semigroup, and let $\Delta_S:C(S)\rightarrow C(S\times S)$ be
the canonical coproduct given by $\Delta_S(f)(s,t) = f(st)$ for $f\in C(S)$ and
$s,t\in S$.  Then an operator $\theta:C(S)\rightarrow L^\infty(X)$ \emph{intertwines}
the coproducts if $(\theta\otimes\theta)\Delta_S = \Delta\theta$.  This is equivalent to
$\theta^*:L^1(X)\rightarrow M(S)$ being a Banach algebra homomorphism.
If $\theta$ is also a $*$-homomorphism, then we write
$\theta \in \operatorname{Mor}(S,L^\infty(X))$.
The following is now immediate.

\begin{proposition}\label{ab_ap_char}
Let $S$ be a compact semigroup, and let $\theta\in\operatorname{Mor}(S,L^\infty(X))$.
Then the image of $\theta$ is contained in $\ap(L^1(X))$.  Furthermore,
$\ap(L^1(X))$ is the union of the images of all such $\theta$.
\end{proposition}

Let $\mathbb G_1=(L^\infty(X_1),\Delta_1)$ and $\mathbb G_2=(L^\infty(X_2),\Delta_2)$
be commutative Hopf von Neumann algebras.  A \emph{morphism} between $\mathbb G_1$
and $\mathbb G_2$ is a normal unital $*$-homomorphism
$T:L^\infty(X_2)\rightarrow L^\infty(X_1)$ which intertwines the coproducts, that is,
$(T\otimes T) \circ \Delta_2 = \Delta_1 \circ T$.  Again, this is equivalent to the
preadjoint $T_*:L^1(X_1)\rightarrow L^1(X_2)$ being a homomorphism of Banach algebras.

\begin{lemma}\label{homo_to_wap}
Let $\mc A$ and $\mc B$ be Banach algebras, and let $T:\mc A\rightarrow\mc B$
be a homomorphism.  Then $T^*$ maps $\ap(\mc B^*)$ to $\ap(\mc A^*)$, and maps
$\wap(\mc B^*)$ to $\wap(\mc A^*)$.
\end{lemma}
\begin{proof}
This is folklore, and follows by observing that for $\mu\in\mc B^*$, we have
that $R_{T^*(\mu)} = T^* \circ R_\mu \circ T$.
\end{proof}

Given commutative Hopf von Neumann algebras $\mathbb G_1$ and $\mathbb G_2$,
let $\ap(L^\infty(X_i)) = C(K^{(i)}_\ap)$ for $i=1,2$.
Given a morphism $T$ from $\mathbb G_1$ to $\mathbb G_2$, the lemma shows that
$T(\ap(L^\infty(X_2))) \subseteq \ap(L^\infty(X_1))$, and so as $T$ is a
$*$-homomorphism, we get a continuous map $T_\ap:K^{(1)}_\ap \rightarrow K^{(2)}_\ap$.
As $T$ intertwines the coproducts, it follows that $T_\ap$ is a semigroup homomorphism.

\begin{proposition}\label{ap_functor}
The assignment of $K_\ap$ to $(L^\infty(X),\Delta)$, and of $T_\ap$ to $T$,
is a functor between the category of commutative Hopf von Neumann algebras and the
category of compact topological semigroups with continuous homomorphisms.
\end{proposition}
\begin{proof}
The only thing to check is that maps compose correctly; but this is an easy,
if tedious, verification.
\end{proof}

We now specialise to the case when $L^1(X) = M(G)$ for a locally compact group $G$.
Let $G$ and $H$ be locally compact groups, and let $\theta:G\rightarrow H$ be a
continuous group homomorphism.  As $\theta$ need not be proper, we only get an
induced map $\theta_*:C_0(H)\rightarrow C(G)$ given by
\[ \theta_*(f)(s) = f(\theta(s)) \qquad (f\in C_0(H), s\in G). \]
However, we embed $C(G)$ into $M(G)^*$ as in Section~\ref{sec:intro}, which
gives a $*$-homomorphism $T_0:C_0(H)\rightarrow M(G)^*$ which satisfies
\[ \ip{T_0(f)}{\mu} = \int_G f(\theta(s)) \ d\mu(s) \qquad (f\in C_0(H), \mu\in M(G)). \]
We can extend $T_0$ by weak$^*$-continuity to get a normal $*$-homomorphism
$T:M(H)^*\rightarrow M(G)^*$, which is easily checked to be unital.
For $f\in C_0(H)$ and $\mu,\lambda\in M(G)$, we have that
\begin{align*} \ip{T_*(\mu)T_*(\lambda)}{f}
&= \int_{H\times H} f(gh) \ dT_*(\mu)(g) \ dT_*(\lambda)(h)
= \int_H \int_G f(\theta(s)h) \ d\mu(s) \ dT_*(\lambda)(h) \\
&= \int_G \int_G f(\theta(s)\theta(t) \ d\lambda(t) \ d\mu(s)
= \int_{G\times G} f(\theta(st)) \ d\mu(s) \ d\lambda(t)
= \ip{T_*(\mu\lambda)}{f}. \end{align*}
Hence $T_*$ is a Banach algebra homomorphism.

Let us take a diversion briefly, and think about the converse.  That is, let $G$ and
$H$ be locally compact groups, and let $T:M(H)^*\rightarrow M(G)^*$ be a normal
unital $*$-homomorphism such that $T_*:M(G)\rightarrow M(H)$ is a homomorphism.
Then $T_0=T\kappa_{C_0(H)}:C_0(H)\rightarrow M(G)^*$ is a $*$-homomorphism.  Notice that
if $T$ were given by some $\theta$ as in the previous paragraph, then $T_0$ would map
into $C(G)\subseteq M(G)^*$.  For each $g\in G$, $\delta_g\in M(G)$ is a normal
character on $M(G)^*$, and so $C_0(H)\rightarrow\mathbb C; f\mapsto \ip{T_0(f)}{\delta_g}$
is a character (which cannot be zero, as $T$ is unital).  We thus get a map
$\theta:G\rightarrow H$ such that $T_*(\delta_g) = \delta_{\theta(g)}$ for $g\in G$.
It follows that $\theta$ is a homomorphism; and if $T_0$ takes values in $C(G)$,
it follows that $\theta$ is continuous.  We have thus shown that $T$ is associated to
a continuous group homomorphism $G\rightarrow H$ if and only if $T(C_0(H))
\subseteq C(G)$.

Suppose now that $\theta:G\rightarrow H$ is a homomorphism which is Borel measurable.
Then we can find $T_*:M(G)\rightarrow M(H)$ by
\[ \ip{T_*(\mu)}{f} = \int_G f(\theta(s)) \ d\mu(s) \qquad (\mu\in M(G), f\in C_0(H)), \]
as $f\circ\theta$ is Borel, and hence $\mu$-measurable.  It is easy to check that
$T=T_*^*$ is then a unital $*$-homomorphism.  For $\mu,\lambda\in M(G)$ and $f\in C_0(H)$,
\begin{align*} \ip{T_*(\mu) T_*(\lambda)}{f}
&= \int_H f(st) \ dT_*(\mu)(s) \ dT_*(\lambda)(t)
= \int_H \int_G f(\theta(s)t) \ d\mu(s) \ dT_*(\lambda)(t) \\
&= \int_G \int_G f(\theta(s)\theta(t)) \ d\lambda(t) \ d\mu(s)
= \int_{G\times G} f(\theta(st)) \ d(\mu\otimes\lambda)(s,t).
\end{align*}
We identify (by integration) the set of bounded, Borel functions on $G$ with
a $*$-subalgebra of $M(G)^*$.  Thus, we can find a bounded net $(f_\alpha)$
in $C_0(G)$ with $f_\alpha\rightarrow f\circ\theta$ weak$^*$ in $M(G)^*$.
In particular, by applying a point mass, we see that $f_\alpha \rightarrow f\circ\theta$
pointwise.  So by Dominated Convergence,
\[ \ip{T_*(\mu) T_*(\lambda)}{f}
= \lim_\alpha \int_{G\times G} f_\alpha(st) \ d(\mu\otimes\lambda)(s,t).
= \lim_\alpha \ip{\mu\lambda}{f_\alpha}
= \ip{T_*(\mu\lambda)}{f}. \]
Thus $T_*$ is a homomorphism.

Notice that $T(C_0(G)) \subseteq C(H)$ if and only if $\theta$ is actually continuous.
It turns out that, rather often, Borel measurable homomorphisms are already continuous.
By \cite[Theorem~22.18]{HR}, this holds if $H$ is separable or $\sigma$-compact
(see also \cite{neeb} for related results).  However, it seems unlikely that all
morphisms are induced by continuous group homomorphisms.

The following is now easily proved.

\begin{proposition}\label{ap_g_functor}
The assignment of $K_\ap$ to $(M(G)^*,\Delta)$, and of $T_\ap$ to $\theta$,
defines a functor between the category of locally compact groups with continuous
homomorphisms and compact topological semigroups with continuous homomorphisms.
\end{proposition}

From now on, fix a locally compact group $G$ and a compact topological semigroup $K_\ap$
with $C(K_\ap) = \ap(M(G))$.
The above proposition, in the abstract, tells us that $K_\ap$ depends only upon $G$.
In the remainder of this section, we study some properties of $K_\ap$, with the eventual
aim (not realised yet) of describing $K_\ap$ ``directly'' using $G$.  For example, if $G$
is discrete, then $K_\ap$ is nothing but the usual almost periodic compactification of $G$,
that is, \emph{the} group compactification of $G$.

Let $G_d$ be the group $G$ with the discrete topology.  For each $s\in G$,
the point mass measure $\delta_s\in M(G)$ induces a normal character on $L^\infty(X)$,
and hence by restriction a character on $\ap(M(G)) = C(K_\ap)$.  Hence we get a map
$\theta_0: G_d \rightarrow K_\ap$.

\begin{proposition}
The map $\theta_0:G_d\rightarrow K_\ap$ is a semigroup homomorphism sending the unit
of $G_d$ to the unit of $K_\ap$.
\end{proposition}
\begin{proof}
Let $f\in C(K_\ap) = \ap(M(G))$ so that for $s,t\in G$,
\[ f(\theta_0(st)) = \ip{f}{\delta_{st}} = \ip{f}{\Delta_*(\delta_s\otimes\delta_t)}
= \ip{\Delta(f)}{\delta_s\otimes\delta_t)} = f(\theta_0(s) \theta_0(t)). \]
This is enough to show that $\theta_0(st) = \theta_0(s)\theta_0(t)$, as required.
Finally, let $e\in G_d$ be the unit.  Then $\delta_e\in M(G)$ is the unit of the
Banach algebra $M(G)$, and so the image of $\delta_e$ in $\ap(M(G))^*$ is a unit.
It follows that $\theta_0(e)$ is a unit for $K_\ap$.
\end{proof}

Following \cite[Section~4.1]{BJM},
for example, let $\ap(G_d) \subseteq \ell^\infty(G)$ be the space of almost periodic
functions on $G_d$.  Then $\ap(G_d)$ is a commutative C$^*$-subalgebra of
$\ell^\infty(G)$ with character space $(G_d)^\ap$, the \emph{almost periodic
compactification} of $G_d$.  As $G_d$ is a group, this agrees with the
\emph{strongly almost periodic} compactification, so that $(G_d)^\ap$ is a group.
This follows easily by extending the inverse from $G_d$, and using that the product
in $(G_d)^\ap$ is jointly continuous.
See \cite[Corollary~4.1.12]{BJM} for further details, for example.

As $K_\ap$ is a topological semigroup,  by the universal property of the almost
periodic compactification, there exists a continuous semigroup homomorphism
$\theta:(G_d)^\ap\rightarrow K_\ap$ making the following diagram commute:
\[ \xymatrix{ G_d \ar[r] \ar[rd]_{\theta_0} & (G_d)^\ap \ar[d]^{\theta} \\
& K_\ap } \]

We regard $\ap(G_d) = C((G_d)^\ap)$ as a subalgebra of $l^\infty(G) = C(G_d)$.
Recall (see \cite[Section~3.3]{dales} for further details) that $M(G) = M_c(G)
\oplus_1 \ell^1(G)$, where $\ell^1(G)$ is identified with the atomic measures
in $M(G)$, and $M_c(G)$ is the space of non-atomic measures.  Then $M_c(G)$ is
an ideal in $M(G)$, and so the projection $P:M(G) \rightarrow \ell^1(G)$ is
an algebra homomorphism.

\begin{lemma}\label{discrete_lemma}
$P^*:\ell^\infty(G)\rightarrow M(G)^*$ is an algebra homomorphism which maps
$\ap(G_d)$ into $\ap(M(G))$.
\end{lemma}
\begin{proof}
Let $\mu\in M(G)$ and $P(\mu) = a = \sum_{s\in G} a_s \delta_s \in \ell^1(G)$.
The product on $M(G)^*$ is simply the Arens product on $C_0(G)^{**}$.  For
$f,g\in C_0(G)$, we have that $\ip{\mu\cdot f}{g} = \int_G f(s) g(s) \ d\mu(s)$
and so $\mu\cdot f = f\mu$ the pointwise product.  It is hence easy to see that
\[ P(\mu\cdot f) = P(f\mu) = \sum_{s\in G} f(s) a_s
= P(\mu)\cdot f. \]
For $\Phi=(\Phi_s)_{s\in G}\in\ell^\infty(G)$, we see that
\[ \ip{P^*(\Phi)\cdot\mu}{f} = \ip{\Phi}{P(\mu\cdot f)}
= \sum_{s\in G} \Phi_s f(s) a_s = \ip{\sum_s \Phi_s a_s \delta_s}{f}. \]
Thus for $\Psi=(\Psi_s)_{s\in G}\in\ell^\infty(G)$,
\[ \ip{P^*(\Psi)P^*(\Phi)}{\mu} = \ip{P^*(\Psi)}{\sum_s \Phi_s a_s \delta_s}
= \sum_{s\in G} \Psi_s \Phi_s a_s = \ip{\Psi\Phi}{P(\mu)}
= \ip{P^*(\Psi\Phi)}{\mu}, \]
showing that $P^*$ is a homomorphism, as required.

As $P$ is a Banach algebra homomorphism $M(G) \rightarrow \ell^1(G)$, by
Lemma~\ref{homo_to_wap}, we have that $P^*$ maps
$\ap(\ell^1(G)) = \ap(G_d)$ into $\ap(M(G))$, as claimed.
\end{proof}

As $P$ is an algebra homomorphism, dualising, we see that
\[ \Delta \circ P^* = (P^*\otimes P^*)\circ \Phi_d, \]
where $\Phi_d:\ell^\infty(G) \rightarrow \ell^\infty(G\times G)$ is the
coproduct for $G_d$.  As $P^*:\ap(G_d)\rightarrow \ap(M(G))=C(K_\ap)$ is a homomorphism,
we get a continuous map $\theta_1 : K_\ap \rightarrow (G_d)^\ap$.  As $P^*$ intertwines
the coproducts, it follows that $\theta_1$ is a semigroup homomorphism.

\begin{lemma}
Consider the continuous semigroup homomorphisms $\theta:(G_d)^\ap\rightarrow K_\ap$
and $\theta_1:K_\ap\rightarrow (G_d)^\ap$.  Then $\theta_1\circ \theta$ is the
identity on $(G_d)^\ap$ and so $\theta$ is a homeomorphism onto its range.
\end{lemma}
\begin{proof}
For $s\in G$ and $F\in\ap(G_d)$, we calculate that
\[ F\big( \theta_1\theta_0(s) \big) = \ip{\delta_{\theta_0(s)}}{P^*(F)}
= \ip{P^*(F)}{\delta_s} = \ip{F}{\delta_s} = F(s). \]
Hence $\theta_1 \circ \theta_0: G \rightarrow (G_d)^\ap$ is the canonical inclusion.
By continuity, it follows that $\theta_1 \circ \theta$ is the identity on
$(G_d)^\ap$, and so $\theta$ must be a homeomorphism onto its range.
\end{proof}

We now prove a simple fact about semigroups: this is surely a folklore result.

\begin{lemma}\label{semigp_lemma}
Let $K$ be a semigroup, let $H$ be a group, let $\theta:H\rightarrow K$
and $\psi:K\rightarrow H$ be semigroup homomorphisms with $\psi\theta$ the
identity on $H$ and $\theta(e_H)$ a unit for $K$.
Let $K_0$ be the kernel of $\psi$, so $K_0 = \psi^{-1}(\{e_H\})$.
Then $K = H \ltimes K_0$ as a semigroup.

Furthermore, if $K$ is a topological semigroup, $H$ is a topological group,
and $\theta$ and $\psi$ are continuous, then $K = H \ltimes K_0$
as topological semigroups.
\end{lemma}
\begin{proof}
Let $H$ act on $K_0$ by
\[ s\cdot k = \theta(s) k \theta(s^{-1}) \qquad (s\in H, k\in K_0). \]
As $\psi(\theta(s)k\theta(s^{-1})) = s \psi(k) s^{-1} = s e_H s^{-1} = e_H$,
it follows that $s\cdot k\in K_0$ as claimed.
Then $H \ltimes K_0$ is the set $H\times K_0$ with the semigroup product
\[ (s,k)(t,l) = (st, k(s\cdot l)) \qquad (s,t\in H, k,l\in K_0). \]
We define a map $\phi:H\ltimes K_0\rightarrow K$ by $\phi(s,k) = k\theta(s)$.
Then
\[ \phi\big( (s,k)(t,l) \big) = k \theta(s) l \theta(s^{-1}) \theta(st)
= k \theta(s) l \theta(t) = \phi(s,k) \phi(t,l), \]
so $\phi$ is a semigroup homomorphism.  If $\phi(s,k) = \phi(t,l)$ then
$k\theta(s) = l\theta(t)$, and so $s = \psi(k\theta(s)) = \psi(l\theta(t)) = t$
and $k = k\theta(e_H) = l\theta(ts^{-1}) = l\theta(e_H)=l$.  Hence $\phi$ is injective.
A calculation shows that for $k\in K$, $k\theta(\psi(k)^{-1}) \in K_0$ and
$\phi(\psi(k), k\theta(\psi(k)^{-1})) = k$, so $\phi$ is a bijection, as required.

When $K$ and $H$ are topological and $\theta$ and $\psi$ are continuous, then
$K_0$ is a closed sub-semigroup of $K$.  The action of $H$ on $K_0$ is
continuous (by joint continuity) and $\phi$ is continuous, as required.
\end{proof}

In our situation, we immediately see the following.

\begin{corollary}
Form the maps $\theta:(G_d)^\ap\rightarrow K_\ap$ and
$\theta_1:K_\ap\rightarrow (G_d)^\ap$ as above.  Let $K_0$ be the kernel
of $\theta_1$.  Then $K_\ap = (G_d)^\ap \ltimes K_0$.
\end{corollary}

By importing some results of \cite{dgh} relating to derivations, we can show that
$K_0$ is not trivial.

\begin{proposition}\label{kzero_non_triv}
For a non-discrete group $G$, we have that $K_0$ is a non-trivial semigroup.
\end{proposition}
\begin{proof}
We have the augmentation character (see \cite[Definition~3.3.29]{dales})
\[ \varphi:\ell^1(G)\rightarrow\mathbb C; \quad
\sum_{s\in G} a_s \delta_s \mapsto \sum_{s\in G} a_s, \]
which, as $M(G) = M_c(G) \oplus \ell^1(G)$, has an extension
$\tilde\varphi:M(G)\rightarrow\mathbb C$ given by
\[ \tilde\varphi(a \oplus \mu) = \varphi(a) \qquad ( a\in\ell^1(G), \mu\in M_c(G)). \]
It is shown in \cite[Theorem~3.2]{dgh} that if $G$ is non-discrete, then there
is a non-zero continuous point derivation at $\tilde\varphi$.  That is, there
exists a non-zero $\Phi\in M(G)^*$ with
\[ \ip{\Phi}{\mu\lambda} = \tilde\varphi(\mu) \ip{\Phi}{\lambda}
+ \ip{\Phi}{\mu} \tilde\varphi(\lambda)
\qquad (\mu,\lambda \in M(G)). \]
Indeed, the proof proceeds as follows.  There exists a non-zero, translation
invariant $\Phi\in M(G)^*$ such that $\ip{\Phi}{\mu}=0$ for $\mu\in\ell^1(G)$
or $\mu\in M_c(G)^2$.  That $\Phi$ is a point derivation follows by a calculation.

It follows immediately that
\[ \Delta(\Phi) = \Phi\otimes\tilde\varphi + \tilde\varphi\otimes\Phi, \]
so that $\Phi \in \ap(M(G))$.  Suppose towards a contradiction that
$\Phi = P^*(\Psi)$ for some $\Psi\in\ap(G_d)$.  Then
\[ \ip{\Phi}{a\oplus\mu} = \ip{\Psi}{a} = \ip{\Phi}{a\oplus 0} = 0
\qquad (a\in\ell^1(G), \mu\in M_c(G)), \]
giving a contraction.  Hence $P^*(\ap(G_d)) \subsetneq \ap(M(G))$, and so
$K_\ap$ is strictly larger than $(G_d)^\ap$; equivalently, $K_0$ is non-trivial.
\end{proof}

Suppose that $G$ is abelian, so that $K_\ap$ is also abelian.
By \cite[Theorem~2.8]{murtus}, as $(C(K_\ap), \Delta)$ is a \emph{quantum semigroup},
we have that $C(K_\ap)$ admits a ``Haar state'', that is, there exists $\mu\in C(K_\ap)^*=
M(K_\ap)$ such that
\[ (\mu\otimes\id)\Delta(F) = (\id\otimes\mu)\Delta(F) = \ip{\mu}{F} 1
\qquad (F\in C(K_\ap)). \]
For $t\in K$, by applying $\delta_t$, we see that
\[ \int_K F(st) \ d\mu(s) = \int_K F(ts) \ d\mu(s) = \int_K F(s) \ d\mu(s)
\qquad (F\in C(K_\ap)). \]
Let $\lambda$ be the image of $\mu$ under $\theta_1$, so that
\[ \ip{\lambda}{f} = \int_K f(\theta_1(s)) \ d\lambda(s) \qquad
(f\in C((G_d)^\ap) = \ap(G_d)). \]
A simple calculation shows that $\lambda$ is the Haar measure on $(G_d)^\ap$.

As shown after \cite[Theorem~2.8]{murtus}, it is not true that $C(L)$ always
carries an invariant probability measure, for a compact semigroup $L$.
It would be interesting to know if $C(K_\ap) = \ap(M(G))$ always carries an
invariant probability measure.

\subsection{Structure semigroup}

Let $\p(M(G))$ be the closure of the collection of $F\in M(G)^*$ such that $\Delta(F)$
is a (finite-rank) tensor in $M(G)^* \otimes M(G)^*$.  This is easily seen to
be a C$^*$-subalgebra of $M(G)^*$, and an $M(G)$-submodule of $M(G)^*$.
Repeating the argument of Theorem~\ref{get_top_semi} yields that
$\p(M(G)) = C(K_\p)$ for some topological semigroup $K_\p$.

Taylor introduced the \emph{structure semigroup} of $G$ in \cite{taylor}.
We shall follow the presentation of \cite{DLS} instead, and define $\Phi = \Phi_{M(G)}$ to
be the character space of $M(G)$.  In our language, $F\in\Phi_{M(G)} \subseteq M(G)^*$
if and only if $\Delta(F)=F\otimes F$.  Let $X_G$ be the closed linear span of $\Phi$
in $M(G)^*$.  Then $X_G\subseteq \p(M(G))$, and again, it can be shown that
$X_G$ is a C$^*$-subalgebra of $M(G)^*$, and an $M(G)$-submodule of $M(G)^*$.
Then the structure semigroup of $G$, written $S(G)$, is the spectrum of $X_G$,
which is again a topological semigroup.

It is asked in \cite{DLS} (in the abelian case) whether $S(G) = K_\ap$.  We can split
this into two questions.  Firstly, if $G$ is abelian, does it follow that
$S(G) = K_\p$, or equivalently, that $X_G = \p(M(G))$?  This is true for a discrete
group $G$, essentially because of the Peter-Weyl theorem, and Fourier analysis,
applied to the compact abelian group $G^\ap$.

Secondly, for a general $G$, do we have that $K_\p = K_\ap$, or equivalently, that
$\p(M(G)) = \ap(M(G))$?  For this question, consider $F\in\ap(M(G))$.  Then by
definition, $\Delta(F) : L^1(X) \rightarrow L^\infty(X)$ is compact.  As $L^\infty(X)$
has the approximation property, it follows that there is a sequence $(T_n)$ of
finite-rank maps $L^1(X)\rightarrow L^\infty(X)$, such that $T_n \rightarrow \Delta(F)$.
Then $K_\p = K_\ap$ if and only if we can always choose the $T_n$ to be of the form
$\Delta(F_n)$ (so that $\Delta(F_n)$ is finite-rank, that is, $F_n \in P(M(G))$).
Again, in the discrete case, this follows from the Peter-Weyl theorem.

\subsection{The antipode}\label{antipode}

Let $(L^\infty(X),\Delta)$ be a commutative Hopf von Neumann algebra.  We shall
call a normal $*$-homomorphism $R:L^\infty(X)\rightarrow L^\infty(X)$ an
\emph{antipode} if $R^2=\id$ and $\Delta R = (R\otimes R) \chi \Delta$, where
$\chi:L^\infty(X\times X) \rightarrow L^\infty(X\times X)$ is the swap map,
$\chi(F)(s,t) = F(t,s)$, for $F\in L^\infty(X\times X)$ and $s,t\in X$.

For example, consider $(C_0(G),\Phi)$ for a locally compact group $G$.  Then
we define $r:C_0(G)\rightarrow C_0(G)$ by $r(f)(s) = f(s^{-1})$ for $f\in C_0(G)$
and $s\in G$.  Then $r$ is an antipode, if we extend the definition to C$^*$-algebras
in the obvious way.  Let $(L^\infty(X),\Delta)$ be induced by $(C_0(G),\Phi)$
as before, so that $L^1(X)=M(G)$.  Define $R_*:L^1(X)\rightarrow L^1(X)$ to be the
map $r^*$, and let $R=R_*^*$.  Then $R$ is a normal $*$-homomorphism, and $R^2=\id$.
For $a,b\in M(G)$ and $f\in C_0(G)$, we see that, as $\Delta_*$ induces the usual
convolution product on $M(G)$,
\[ \ip{r^*\Delta_*(a\otimes b)}{f} = \ip{a\otimes b}{\Delta r(f)}
= \ip{b\otimes a}{(r\otimes r)\Delta(f)}
= \ip{\Delta_*\chi(r^*\otimes r^*)(a\otimes b)}{f}. \]
Hence $R_*\Delta_* = \Delta_* \chi (R_*\otimes R_*)$.  So, for $F\in L^\infty(X)$
and $a,b\in L^1(X)=M(G)$, we see that
\begin{align*} \ip{\Delta R(F)}{a\otimes b} &= \ip{F}{R_*\Delta_*(a\otimes b)}
= \ip{F}{\Delta_*\chi(R_*\otimes R_*)(a\otimes b)} \\
&= \ip{(R\otimes R)\chi\Delta(F)}{a\otimes b}. \end{align*}
Hence $R$ is an antipode on $(L^\infty(X),\Delta)$.

\begin{lemma}
Let $(L^\infty(X),\Delta)$ be a commutative Hopf von Neumann algebra, equipped
with an anitpode $R$.  Then $R$ restricts to give $*$-homomorphisms on
$\ap(L^1(X))$ and $\wap(L^1(X))$.
\end{lemma}
\begin{proof}
We know that $F\in\ap(L^1(X))$ if and only if $\Delta(F)\in L^\infty(X)
\inten L^\infty(X)$.  Hence, for $F\in\ap(L^1(X))$, we see that $\Delta R(F)
= (R\otimes R)\chi\Delta(F) \in L^\infty(X) \inten L^\infty(X)$, and so
$R(F)\in\ap(L^1(X))$, as required.

Now suppose that $F\in\wap(L^1(X))$, so $\Delta(F) : L^1(X) \rightarrow L^\infty(X)$
is weakly-compact.  Then, for $a,b\in L^1(X)$,
\begin{align*} \ip{\Delta R(F)(a)}{b} &= \ip{(R\otimes R)\chi\Delta(F)}{a\otimes b}
= \ip{\Delta(F)}{R_*(b) \otimes R_*(a)} \\
&= \ip{R\Delta(F)^*\kappa_{L^1(X)}R_*(a)}{b}. \end{align*}
Thus $\Delta R(F) = R\Delta(F)^*\kappa_{L^1(X)}R_*$, which is weakly-compact
if $\Delta(F)$ is, as required.
\end{proof}

Hence $R$ induces an \emph{involution} on $K_\ap$, written $s\mapsto s'$.  This
means that $(st)' = t' s'$ for $s,t\in K$, and $R(F)(s) = F(s')$ for
$F\in\ap(M(G))=C(K_\ap)$ and $s\in K_\ap$.  There is no reason to expect this to
be an inverse map on $K_\ap$, but we do have the following.

\begin{proposition}
Consider the map $\theta:(G_d)^{\ap} \rightarrow K_\ap$ as above, and recall that
$(G_d)^\ap$ is a (compact) group.  Then $\theta(s^{-1}) = \theta(s)'$ for
$s\in(G_d)^\ap$.
\end{proposition}
\begin{proof}
Recall that, because of joint continuity, the inverse in $(G_d)^\ap$ satisfies
the following property.  Let $s\in (G_d)^\ap$, so we can find a net $(s_\alpha)$
in $G_d$ which converges to $s$.  By possibly moving to a subnet, we have that
$s^{-1} = \lim_\alpha s_\alpha^{-1}$.  Now let $F\in\ap(M(G))$, so that
\begin{align*} F(\theta(s)') &= R(F)(\theta(s)) = \lim_\alpha R(F)(\theta_0(s_\alpha))
= \lim_\alpha \ip{R(F)}{\delta_{s_\alpha}} 
= \lim_\alpha \ip{F}{r^*(\delta_{s_\alpha})} \\
&= \lim_\alpha \ip{F}{\delta_{s_\alpha^{-1}}}
= \lim_\alpha F(\theta_0(s_\alpha^{-1}))
= F(\theta(s^{-1})), \end{align*}
as required.
\end{proof}

We have hence demonstrated various properties of the compact semigroup $K_\ap$.
These do not, however, appear to be enough to characterise $K_\ap$ directly, just
in terms of $G$.

\section{Weakly almost periodic functionals}\label{wap_case}

For a commutative Hopf von Neumann algebra $(L^\infty(X),\Delta)$, we know that
$\wap(L^1(X))$ is a unital commutative C$^*$-algebra, say $C(K_\wap)$.  In this section,
we shall show that $K_\wap$ is a compact \emph{semitopological semigroup}, that
is, a semigroup whose product is separately continuous.  This is in complete
agreement for what happens with $L^1(G)$, see \cite{ulger} and \cite[Section~4.2]{BJM}.

\subsection{Embedding spaces of separately continuous functions}

Let $L^\infty(X)$ be a commutative von Neumann algebra, and let $\Delta:L^\infty(X)
\rightarrow L^\infty(X\times X)$ be a co-associative normal $*$-homomorphism,
turning $L^1(X)$ into a Banach algebra.  We can find a compact, Hausdorff,
hyperstonian space $K$ such that $L^\infty(X) = C(K)$, see, for example,
\cite[Section~1, Chapter~III]{tak}.  Notice, however, that $L^\infty(X\times X)$
is, in general, much larger than $C(K\times K)$.

Let  $SC(K\times K)$ be the space of functions $K\times K\rightarrow\mathbb C$
which are separately continuous.  Obviously $SC(K\times K)$ is a C$^*$-algebra.
For $f\in SC(K\times K)$ and $\mu\in M(K)$,
define functions $(\mu\otimes\iota)f, (\iota\otimes\mu)f :K\rightarrow\mathbb C$ by
\[ (\mu\otimes\iota)f (k) = \int_K f(l,k) \ d\mu(l), \quad
(\iota\otimes\mu)f (k) = \int_K f(k,l) \ d\mu(l) \qquad (k\in K). \]
It is shown in \cite[Lemma~2.2]{runde} (using a result of Grothendieck) that
actually $(\mu\otimes\iota)f$ and $(\iota\otimes\mu)f$ are in $C(K)$.

Then \cite[Lemma~2.4]{runde} shows that
\[ \ip{(\mu\otimes\iota)f}{\lambda} = \ip{(\iota\otimes\lambda)f}{\mu}
\qquad (f\in SC(K\times K), \mu,\lambda\in M(K)). \]
We write $\ip{\mu\otimes\lambda}{f}$ for this.  Furthermore, \cite[Lemma~2.4]{runde}
shows that the map $M(K)\times M(K)\rightarrow\mathbb C; (\mu,\lambda)\mapsto
\ip{\mu\otimes\lambda}{f}$ is separately weak$^*$-continuous in each variable.
These results rely upon \cite{johnson}, which shows that each $f\in SC(K\times K)$
is $\mu$-measurable for any $\mu\in M(K\times K)$.

For $a\in L^1(X)$, we have that $\kappa_{L^1(X)}(a) \in L^\infty(X)^* = C(K)^*$ and hence
induces a measure, say $\mu_a \in M(K)$.  Define a map
$\theta_{sc}:SC(K\times K)\rightarrow L^\infty(X\times X) = (L^1(X)\proten L^1(X))^*$ by
\[ \ip{\theta_{sc}(f)}{a\otimes b} = \ip{\mu_a\otimes \mu_b}{f}
\qquad (f\in SC(K\times K), a,b\in L^1(X)). \]

\begin{proposition}
The map $\theta_{sc}$ is an isometric $*$-homomorphism.
\end{proposition}
\begin{proof}
Clearly $\theta_{sc}$ is a contraction.  For $k\in K$, let $\delta_k\in M(K)=L^1(X)^{**}$
be the point-mass at $k$, so that $\delta_k$ is the weak$^*$-limit of norm-one elements
of $L^1(X)$, say $a^{(k)}_\alpha \rightarrow \delta_k$.  By separate weak$^*$-continuity,
for $f\in SC(K\times K)$ and $k,l\in K$,
\[ f(k,l) = \lim_\alpha \lim_\beta \ip{\mu_{a^{(k)}_\alpha} \otimes\mu_{a^{(l)}_\alpha}}{f}
= \lim_\alpha \lim_\beta \ip{\theta_{sc}(f)}{a^{(k)}_\alpha \otimes a^{(l)}_\alpha}.\]
By taking the supremum over all $k$ and $l$, this shows that
$\theta_{sc}$ is an isometry.

To show that $\theta_{sc}$ is a $*$-homomorphism, we argue as follows.
Let $f\in SC(K\times K)$ and $x,y\in L^\infty(X)=C(K)$, and set $g=x\otimes y$.
Such $g$ are linearly dense in $C(K\times K)$ and hence separate the points of
$M(K)\proten M(K)$.  We also regard $g$ as a member of $L^\infty(X\times X)$.
Let $a,b\in L^1(X)$ and consider $\theta_{sc}(f)(a\otimes b)$, defined as usual by
\[ \ip{F}{\theta_{sc}(f)(a\otimes b)} = \ip{F \theta_{sc}(f)}{a\otimes b}
\qquad (F\in L^\infty(X\times X)). \]
Then $\theta_{sc}(f)(a\otimes b) \in L^1(X)\proten L^1(X)$, so we can find
sequences $(c_n),(d_n)$ in $L^1(X)$ with $\theta_{sc}(f)(a\otimes b)
= \sum_n c_n \otimes d_n$ and $\sum_n \|c_n\| \|d_n\| < \infty$.  Then
\begin{align*}
\sum_n \ip{x}{c_n} \ip{y}{d_n} &=
\ip{g}{\theta_{sc}(f)(a\otimes b)}
= \ip{\theta_{sc}(f)}{(a\otimes b)g} \\
&= \ip{\theta_{sc}(f)}{xa\otimes yb}
= \ip{\mu_{xa} \otimes \mu_{yb}}{f}. \end{align*}
However, it's easy to see that $\mu_{xa} = x \mu_a$, so
\[ \sum_n \ip{\mu_{c_n}}{x} \ip{\mu_{d_n}}{y} = \int_{K\times K} x(k) y(l) f(k,l)
\ d\mu_a(k) d\mu_b(l). \]
As $x,y\in C(K)$ were arbitrary, we conclude that
\[ \sum_n \mu_{c_n} \otimes \mu_{d_n} = f(\mu_a\otimes\mu_b) \]
as measures on $K\times K$.  Then, for $h \in SC(K\times K)$,
\begin{align*} \ip{\theta_{sc}(h) \theta_{sc}(f)}{a\otimes b}
&= \ip{\theta_{sc}(h)}{\theta_{sc}(f)(a\otimes b)}
= \sum_n \ip{\theta_{sc}(h)}{c_n\otimes d_n} \\
&= \sum_n \ip{\mu_{c_n} \otimes \mu_{d_n}}{h}
= \ip{f \mu_a\otimes\mu_b}{h} = \ip{\mu_a\otimes\mu_b}{hf}
= \ip{\theta_{sc}(hf)}{a\otimes b}. \end{align*}
As $a,b\in L^1(X)$ were arbitrary, this shows that
$\theta_{sc}$ is a homomorphism.  A similar argument establishes that
$\theta_{sc}$ is a $*$-homomorphism.
\end{proof}

We henceforth identify $SC(K\times K)$ with a $*$-subalgebra of $L^\infty(X\times X)$.

\subsection{Application to WAP functionals}

For a Banach algebra $\mc A$ we write $\wap(\mc A)$ for the weakly almost periodic
functionals on $\mc A$, which is a closed $\mc A$-submodule of $\mc A^*$.
As shown in \cite[Lemma~1.4]{ll} (for commutative algebras; the proof readily extends
to the general case, compare \cite[Proposition~2.4]{daws}) the Arens products drop to
a well-defined product on $\wap(\mc A)^*$ which is separately weak$^*$-continuous;
that is, $\wap(\mc A)^*$ is a dual Banach algebra.

\begin{proposition}\label{what_is_wap}
Let $(L^\infty(X),\Delta)$ be a commutative Hopf von Neumann algebra, and let
$L^\infty(X) = C(K)$ as before.  For $F\in L^\infty(X)$, define
$f:K\times K\rightarrow\mathbb C$ by
\[ f(k,l) = \ip{\delta_k \aone \delta_l}{F} \qquad (k,l\in K). \]
If $F\in\wap(L^1(X))$ then $f\in SC(K\times K)$ and $\theta_{sc}(f) = \Delta(F)$.
Conversely, if $\Delta(F)=\theta_{sc}(g)$ for some $g\in SC(K\times K)$, then
$F\in\wap(L^1(X))$, and $f=g$.
\end{proposition}
\begin{proof}
Suppose that $F\in\wap(L^1(X))$.  As $K\rightarrow L^\infty(X)^*=M(K);
k\mapsto\delta_k$ is continuous for the weak$^*$-topology, and the product on $\wap(L^1(X))^*$
is separately weak$^*$-continuous, it follows that $f \in SC(K\times K)$.
We claim that $\theta_{sc}(f) = \Delta(F)$.  Indeed, let
$a\in L^1(X)$, and observe that for $x\in L^\infty(X)=C(K)$,
\[ \ip{x}{a} = \ip{\mu_a}{x} = \int_K \ip{\delta_k}{x} \ d\mu_a(k). \]
Thus, for $a,b\in L^1(X)$, and using that $\wap(\mc A)$ is an $\wap(\mc A)^*$
bimodule,
\begin{align*} \ip{\theta_{sc}(f)}{a\otimes b}
&= \int_{K\times K} \ip{\delta_k \aone \delta_l}{F}
\ d\mu_a(k) d\mu_b(l) = \int_K \int_K \ip{\delta_k}{\delta_l\cdot F} \ d\mu_a(k) d\mu_b(l)
\\ &= \int_K \ip{\delta_l\cdot F}{a} \ d\mu_b(l)
= \int_K \ip{\delta_l}{F\cdot a} \ d\mu_b(l)
= \ip{F\cdot a}{b} = \ip{\Delta(F)}{a\otimes b}, \end{align*}
as required.

Conversely, if $F\in L^\infty(X)$ with $\Delta(F)=\theta_{sc}(g)$ for some
$g\in SC(K\times K)$, then for $a,b\in L^1(X)$,
\[ \ip{(a\otimes\iota)\Delta(F)}{b} = \ip{\theta_{sc}(g)}{a\otimes b}
= \ip{\mu_a\otimes\mu_b}{g} = \ip{\mu_b}{(\mu_a\otimes\iota)g}, \]
so that $(a\otimes\iota)\Delta(F) = (\mu_a\otimes\iota)g \in C(K) = L^\infty(X)$.
Thus, for $\mu\in L^\infty(X)^*=M(K)$,
\[ \ip{\mu}{(a\otimes\iota)\Delta(F)} = \ip{\mu}{(\mu_a\otimes\iota)g}
= \ip{\mu_a}{(\iota\otimes\mu)g}. \]
Now let $(a_\alpha)$ be a bounded net in $L^1(X)$.  By moving to a subnet,
we may suppose that $\mu_{a_\alpha}\rightarrow \lambda \in M(K)$ weak$^*$.
For $\mu\in M(K)$, we have
\[ \lim_\alpha \ip{\mu}{(a_\alpha\otimes\iota)\Delta(F)}
= \lim_\alpha \ip{\mu_{a_\alpha}}{(\iota\otimes\mu)g}
= \ip{\lambda}{(\iota\otimes\mu)g}
= \ip{\mu}{(\lambda\otimes\iota)g}, \]
so we see that $(a_\alpha\otimes\iota)\Delta(F) \rightarrow (\lambda\otimes\iota)g$ weakly.
Thus $F\in\wap(L^1(X))$.  As $\theta_{sc}$ is injective, it follows that $f=g$.
\end{proof}

It is worth making a link with Theorem~\ref{get_top_semi}.  Firstly, note that a simple
check shows that $\theta_{sc}$ extends the natural embedding of $C(K\times K) = L^\infty(X)
\inten L^\infty(X)$ into $L^\infty(X\times X)$.  If $F\in\ap(L^1(X))$, then
as the product on $\ap(L^1(X))^*$ is jointly continuous, it follows that $f$ will also
be jointly continuous, so that $f\in C(K\times K) \subseteq SC(K\times K)$.
Conversely, if $\Delta(F) = \theta_{sc}(g)$ for some $g\in C(K\times K)$, then $f=g$
and by the same weak$^*$-approximation argument as used above, it follows that $F$ is
almost periodic.

The following is now immediate!

\begin{theorem}
Let $(L^\infty(X),\Delta)$ be a commutative Hopf von Neumann algebra.
Then $\wap(L^1(X))$ is a C$^*$-algebra.
\end{theorem}
\begin{proof}
The previous Proposition shows that $F\in\wap(L^1(X))$ if and only if
$\Delta(F)\in \theta_{sc}(SC(K\times K))$.  As $\Delta$ is a $*$-homomorphism,
and $\theta_{sc}(SC(K\times K))$ is a C$^*$-subalgebra of $L^\infty(X)$,
it follows that $\wap(L^1(X))$ is a C$^*$-algebra.
\end{proof}

However, we can now easily prove more about the structure of $\wap(L^1(X))$.

\begin{theorem}\label{wap_semitop}
Let $(L^\infty(X),\Delta)$ be a commutative Hopf von Neumann algebra, and let
$K_\wap$ be the character space of $\wap(L^1(X))$.
The map $\Delta$, which restricts to a map $\wap(L^1(X)) \rightarrow
\theta_{sc}(SC(K\times K))$, induces a $*$-homomorphism
\[ \Delta_\wap:C(K_\wap) \rightarrow SC(K_\wap\times K_\wap). \]
The adjoint $\Delta_\wap^*:M(K_\wap)\proten M(K_\wap)\rightarrow M(K_\wap)$
is just the product on $\wap(L^1(X))^*$.  Furthermore, $\Delta_\wap$ induces
a separately continuous (that is, semitopological) semigroup product on $K_\wap$.
At the level of Banach algebras, this product ``is'' the Arens product.
\end{theorem}
\begin{proof}
Let $F\in\wap(L^1(X))$, and let $f\in SC(K\times K)$ with $\theta_{sc}(f)=\Delta(F)$.
Then, for $k,l\in K$,
\[ f(k,l) = \ip{\delta_k\aone\delta_l}{F}
= \ip{\delta_k}{(\iota\otimes\delta_l)f}
= \ip{\delta_l}{(\delta_k\otimes\iota)f}. \]
So, with reference to the proof above,
$(\delta_k\otimes\iota)f = F\cdot\delta_k \in \wap(L^1(X))$ and
$(\iota\otimes\delta_k)f = \delta_k\cdot F \in \wap(L^1(X))$.

Hence $(\delta_k\otimes\iota)f, (\iota\otimes\delta_k)f$ and are members of
$C(K_{\wap})$ for each $k\in K$.  The inclusion $\wap(L^1(X)) =C(K_\wap) \rightarrow C(K)$
induces a continuous surjection $j:K\rightarrow K_{\wap}$.  We claim that we can define
$f_0 \in SC(K_{\wap}\times K_{\wap})$ by
\[ f_0(j(k), j(l)) = f(k,l) \qquad (k,l\in K). \]
Indeed, this is well-defined, for if $j(k)=j(k')$ and $j(l)=j(l')$ then
\begin{align*} f(k,l) &= \ip{\delta_l}{(\delta_k\otimes\iota)f}
= \ip{\delta_{l'}}{(\delta_k\otimes\iota)f} = f(k,l')
= \ip{\delta_k}{(\iota\otimes\delta_{l'})f} \\
&= \ip{\delta_{k'}}{(\iota\otimes\delta_{l'})f} = f(k',l'). \end{align*}
That $f_0$ is separately continuous is immediate, as the same is true of $f$,
and using that $j$ is a closed map.

Denote $f_0$ by $\Delta_\wap(F)$, so that $\Delta_\wap$ is a linear map
$C(K_\wap) \rightarrow SC(K_\wap\times K_\wap)$.
The map $f\mapsto f_0$ is clearly a
$*$-homomorphism, and as $f = \theta_{sc}^{-1}\Delta(F)$, it follows that
$\Delta_\wap$ is also a $*$-homomorphism.  So we have
\[ \Delta_{\wap}: C(K_{\wap}) \rightarrow SC(K_{\wap}\times K_{\wap}), \]
a $*$-homomorphism.  The adjoint $\Delta_\wap^*$ induces a map $M(K_\wap)\proten
M(K_\wap)\rightarrow M(K_\wap)$, and this is simply the Arens product on
$\wap(L^1(X))^* = M(K_\wap)$; in this sense, we could say that $\Delta_\wap$ is
coassociative.  In particular, for $k,l\in K_\wap$, we have that $\delta_k\aone\delta_l
= \Delta_\wap^*(\delta_k\otimes\delta_l)$ is a character on $C(K_\wap)$, and hence
is identified with a point in $K_\wap$.  So $K_\wap$ carries a product, and it is
easy to see that this is associative.  As the product on $M(K_\wap)$ is separately
continuous, the semigroup product is separately continuous.
\end{proof}

The previous result is, to the author, still surprising, for the following
reason.  Then fact that $L^\infty(X)=C(K)$ seems, naively, to be of little use,
as the coproduct $\Delta$ is absolutely not (in general) associated with any
product on $K$ (indeed, \cite[Section~8]{DLS} shows that the (first) Arens product
never induces a product on $K$, unless $G$ is discrete).
Hence, one might expect not to get far working with function spaces; nevertheless,
this is exactly the approach which has worked above.

We now explore the weakly almost periodic version of Proposition~\ref{ab_ap_char}.
Let $S$ be a compact, semitopological semigroup, and let $\Delta_S:C(S)\rightarrow
SC(S\times S)$ be the canonical coproduct, given by $\Delta_S(f)(s,t)=f(st)$ for
$f\in C(S)$ and $s,t\in S$.  Now let $\theta:C(S)\rightarrow L^\infty(X)=C(K)$ be
a unital $*$-homomorphism, so we have an induced continuous map $\theta_*:K\rightarrow
S$.  Then $\theta\otimes\theta:SC(S\times S)\rightarrow SC(K\times K)$ is defined to
be the map $(\theta\otimes\theta)f(k,l) = f(\theta_*(k),\theta_*(l))$ for
$f\in SC(S\times S)$ and $k,l\in K$.  We can now say that $\theta$ \emph{intertwines}
the coproducts if $\Delta\theta = \theta_{sc}(\theta\otimes\theta)\Delta_S$.
Again, this is equivalent to $\theta^*:L^1(X)\rightarrow M(S)$ being a Banach algebra
homomorphism, and we write $\theta\in\operatorname{Mor}(S,L^\infty(X))$ in this case.

\begin{proposition}
Let $S$ be a compact semitopological semigroup, and let $\theta\in\operatorname{Mor}(S,L^\infty(X))$.
Then the image of $\theta$ is contained in $\wap(L^1(X))$.  Furthermore,
$\wap(L^1(X))$ is the union of the images of all such $\theta$.  In particular,
$K_\wap$ is the largest quotient of $K$ which is a semitopological semigroup with
the product induced by $\Delta$.
\end{proposition}
\begin{proof}
By Proposition~\ref{what_is_wap}, and the definition of $\theta\otimes\theta$, it
is immediate that $\theta$ maps into $\wap(L^1(X))$.  Taking $S=K_\wap$ and
$\theta$ to be the inclusion, we see that $\wap(L^1(X))$ arises as the image of
$\theta$.
\end{proof}

Let $\mathbb G_1=(L^\infty(X_1),\Delta_1)$ and $\mathbb G_2=(L^\infty(X_2),\Delta_2)$
be commutative Hopf von Neumann algebras, and let $T:\mathbb G_1\rightarrow\mathbb G_2$
be a morphism.  For $i=1,2$ let $\wap(L^1(X_i)) = C(K^{(i)}_\wap)$, so that
$K^{(i)}_\wap$ is a compact semitopological semigroup.  
By Lemma~\ref{homo_to_wap}, $T$ maps $\wap(L^1(X_2))=C(K^{(2)}_\wap)$ to
$\wap(L^1(X_1))=C(K^{(1)}_\wap)$ and is a $*$-homomorphism, and so induces a continuous
map $T_\wap:K^{(1)}_\wap\rightarrow K^{(2)}_\wap$.

\begin{proposition}\label{wap_functor}
The assignment of $K_\wap$ to $(L^\infty(X),\Delta)$, and of $T_\wap$ to $T$,
defines a functor between the category commutative Hopf von Neumann algebras and the
category of compact semitopological semigroups with continuous homomorphisms.
\end{proposition}
\begin{proof}
We first show that $T_\wap$ is indeed a homomorphism.  With reference to the proof of
Theorem~\ref{wap_semitop}, for $s,t\in K^{(1)}_\wap$, we have that $\delta_{st} =
\delta_s \aone \delta_t$.  As $T^* = T_*^{**}$, it is easy to see that $T^*:
\wap(L^1(X_1))^* \rightarrow \wap(L^1(X_2))^*$ is a homomorphism, so that
$\delta_{T_\wap(st)} = T^*(\delta_{st}) = T^*(\delta_s \aone \delta_t)
= T^*(\delta_s) \aone T^*(\delta_t) = \delta_{T_\wap(s)T_\wap(t)}$,
which shows that $T_\wap$ is a homomorphism.

It is now an easy, though tedious, check that we have defined a functor.
\end{proof}

\section{For the measure algebra}\label{measure_wap_props}

Let $G$ be a locally compact group, and consider $M(G) = L^1(X)$ as the predual
of a commutative Hopf von Neumann algebra.  By applying the results of the previous
sections, we see that $\wap(M(G)) = C(K_\wap)$ for some compact Hausdorff space $K_\wap$,
and that $K_\wap$ becomes a semitopological semigroup in a canonical fashion.
The following now follows in exactly the same way as Proposition~\ref{ap_g_functor}
(where, again, given a continuous group homomorphism $\theta$, we define the
associated morphism $T$ and thus get $T_\wap$ as above).

\begin{proposition}\label{wap_g_functor}
The assignment of $K_\wap$ to $G$, and of $T_\wap$ to $\theta$, is a functor between
the category of locally compact spaces with continuous homomorphisms and compact
semitopological semigroups with continuous homomorphisms.
\end{proposition}

As $M(G)$ is a dual Banach algebra with predual $C_0(G)$, we have that
$C_0(G)\subseteq\wap(M(G))$; see \cite[Section~2]{daws} and references therein.
Clearly $1\in \wap(M(G))$.  So the inclusion $\iota:C_0(G) \rightarrow C(K_\wap)$ factors through
\[ \xymatrix{ C_0(G) \ar[r] & C(G_\infty) \ar[r]^{\iota^\infty} & C(K_\wap) }, \]
where $G_\infty$ is the one-point compactification of $G$.  We can turn $G_\infty$
into a semigroup by letting the added point $\infty$ be a \emph{semigroup zero},
so $s\infty = \infty s = \infty$ for $s\in G$.  Then $G_\infty$ is semitopological,
for if $s_\alpha \rightarrow\infty$, then by definition, for each compact set
$K\subseteq G$, there exists $\alpha_0$ with $s_\alpha\not\in K$ for
$\alpha\geq\alpha_0$.  So for $t\in G$, as $s_\alpha t\in K$ if and only if
$s_\alpha \in K t^{-1}$, and $Kt^{-1}$ is compact, it follows that $s_\alpha t
\rightarrow\infty$.  Similarly $ts_\alpha\rightarrow\infty$.

Thus we have an induced continuous map
$\iota^\infty_* : K_\wap\rightarrow G_\infty$, which has dense and closed range,
and hence must be surjective.

\begin{proposition}\label{iota_homo}
The map $\iota^\infty_*:K_\wap\rightarrow G_\infty$ is a homomorphism.
\end{proposition}
\begin{proof}
Let $\kappa:C(G_\infty)\rightarrow M(G)^*$ be the canonical map, and let $K$
be the compact space such that $M(G)^*=C(K)$.  Hence $\kappa:C(G_\infty)
\rightarrow C(K)$ is an injective $*$-homomorphism, and so there exists a
continuous surjection $\phi:K\rightarrow G_\infty$.  Notice then that we have
the following commutative diagrams
\[ \xymatrix{ C(G_\infty) \ar[r]^{\iota^\infty} \ar[rd]_{\kappa} & C(K_\wap)
\ar[d] \\ & C(K) } \qquad
\xymatrix{ G_\infty & K_\wap \ar[l]_{\iota^\infty_*} \\
& K \ar[u]_{j} \ar[ul]^{\phi} } \]
Here $j:K\rightarrow K_\wap$ is as in (the proof of) Theorem~\ref{wap_semitop} above.

Let $f\in C(G_\infty)$,
and define $\alpha:K\times K \rightarrow \mathbb C$ by
\[ \alpha(k,l) = f(\phi(k) \phi(l)) \qquad (k,l\in K). \]
Thus $\alpha \in SC(K\times K)$ as $G_\infty$ is semitopological.
For $l\in K$, let $f_l\in C(G_\infty)$ be defined by $f_l(s) = f(s\phi(l))$
for $s\in G_\infty$.  For $a,b\in M(G)$, we see that
\begin{align*} \ip{\theta_{sc}(\alpha)}{a\otimes b}
&= \ip{\mu_a\otimes\mu_b}{\alpha}
= \int_K \int_K f(\phi(k)\phi(l)) \ d\mu_a(k) \ d\mu_b(l)
= \int_K \ip{\kappa(f_l}{a} \ d\mu_b(l) \\
&= \int_G \int_K f_l(s) \ d\mu_b(l) \ da(s)
= \int_G \int_K f(s\phi(l)) \ d\mu_b(l) \ da(s) \\
&= \int_G \int_G f(st) \ db(t) \ da(s),
\end{align*}
where the final equality comes from repeating the argument.  Thus
\begin{align*} \ip{\theta_{sc}(\alpha)}{a\otimes b} = \ip{\kappa(f)}{ab}
= \ip{\Delta(\kappa(f))}{a\otimes b}. \end{align*}
We conclude that $\theta_{sc}(\alpha) = \Delta(\kappa(f))$.

Now observe that $\Delta_\wap(\iota_\infty(f))$ is the
map $K_\wap\times K_\wap\rightarrow\mathbb C$ given by
\[ (j(k),j(l)) \mapsto \theta_{sc}^{-1}\Delta(\kappa(f)) (k,l)
= \alpha(k,l) = f(\phi(k)\phi(l)) \qquad (k,l\in K). \]
Let $s,t\in K_\wap$, and pick $k,l\in K$ with $j(k)=s, j(l)=t$.  Then, as
$\iota^\infty_* j = \phi$, we have that $f(\iota^\infty_*(s) \iota^\infty_*(t))
= f(\phi(k) \phi(l)) = \Delta_\wap(\iota_\infty(f))(s,t) = \iota_\infty(f)(st)
= f(\iota^\infty_*(st))$.  As $f\in C(G_\infty)$ was arbitrary, we conclude
that $\iota^\infty_*(s) \iota^\infty_*(t) = \iota^\infty_*(st)$.  Thus
$\iota^\infty_*$ is a homomorphism, as required.
\end{proof}

Let $K_0 = (\iota^\infty_*)^{-1}(\{\infty\})$ a closed subset of $K_\wap$.
As $\iota^\infty_*$ is a homomorphism, it follows that $K_0$ is an ideal in $K_\wap$
and that $K_\wap\setminus K_0$ is a locally compact sub-semigroup of $K_\wap$.

Obviously each $s\in G$ induces a normal character $\delta_s$ on $M(G)^*$, and
hence by restriction, a character on $\wap(M(G))$.  So we have a (possibly
discontinuous) map $\theta:G\rightarrow K_\wap$, which we shall henceforth consider
as a map $\theta:G_d\rightarrow K_\wap$.  Let $s,t\in G$ and $F\in\wap(M(G))$, so that
\[ F(\theta(s)\theta(t)) = \ip{\delta_s\aone\delta_t}{F}
= \ip{F}{\delta_s \delta_t} = \ip{F}{\delta_{st}} = F(\theta(st)), \]
so we see that $\theta$ is a homomorphism.

Denote the unit of $G$ by $e_G$.  As $\delta_{e_G}$ is the unit of $M(G)$,
it follows that $\delta_{e_G}$ is also the unit of $\wap(M(G))^*$, and so
$\theta(e_G)$ is the unit of $K$.

\begin{proposition}
Restrict $\iota^\infty_*$ to a homomorphism $K_\wap\setminus K_0 \rightarrow G$.
Let $K_1$ be the kernel of this homomorphism, so that $K_1$ is a closed
sub-semigroup of $K_\wap \setminus K_0$.  Then $\iota^\infty_* \circ \theta$
is the identity on $G_d$ and $\theta$ maps into $K_\wap\setminus K_0$.  In particular,
$K_\wap\setminus K_0 = G_d \ltimes K_1$.
\end{proposition}
\begin{proof}
For $t\in G$ and $f\in C_0(G)$, clearly $\ip{\delta_{\theta(t)}}{\iota(f)}
= f(t)$, showing that $\theta$ takes values in $K\setminus K_0$, and that
$\iota^\infty_*(\theta(t)) = t$, as required.  We now appeal to Lemma~\ref{semigp_lemma}.
\end{proof}

By the universal property for $\wap$, as $K_\wap$ is compact and semitopological,
we have a factorisation
\[ \xymatrix{ G_d \ar[r]^{\theta} \ar[d] & K_\wap \\
(G_d)^\wap \ar[ru]_{\theta^\wap} } \]
Recall that $\theta^\wap$ must satisfy the following property: for
$s\in (G_d)^\wap$, if $(s_\alpha)\subseteq G_d$ is a net with $s_\alpha
\rightarrow s$ in $(G_d)^\wap$, then $\theta(s_\alpha) \rightarrow \theta^\wap(s)$
in $K_\wap$.

We regard $\wap(G_d) = C((G_d)^\wap)$ as a subalgebra of $l^\infty(G) = C(G_d)$.
As before Lemma~\ref{discrete_lemma} we consider the projection
$P:M(G)\rightarrow\ell^1(G)$, which is an algebra homomorphism.
The following has an almost identical proof to that of Lemma~\ref{discrete_lemma}.

\begin{lemma}
$P^*:\ell^\infty(G)\rightarrow M(G)^*$ is an algebra homomorphism which maps
$\wap(G_d)$ into $\wap(M(G))$.
\end{lemma}

Again, we have that $\Delta \circ P^* = (P^*\otimes P^*)\circ \Phi_d$,
where $\Phi_d:\ell^\infty(G) \rightarrow \ell^\infty(G\times G)$ is the
coproduct for $G_d$.  We hence get a continuous semigroup homomorphism
$\theta_1 : K \rightarrow (G_d)^\wap$.

\begin{lemma}
Consider the continuous semigroup homomorphisms $\theta^\wap:(G_d)^\wap
\rightarrow K_\wap$ and $\theta_1:K_\wap \rightarrow (G_d)^\wap$.  Then
$\theta_1\circ\theta^\wap$ is the identity on $(G_d)^\wap$ and so $\theta^\wap$
is a homeomorphism onto its range.
\end{lemma}
\begin{proof}
For $s\in G$ and $F\in\wap(G_d)$, we calculate that
\[ F\big( \theta_1\theta(s) \big) = \ip{\delta_{\theta(s)}}{P^*(F)}
= \ip{P^*(F)}{\delta_s} = \ip{F}{\delta_s} = F(s). \]
Hence $\theta_1 \circ \theta: G \rightarrow (G_d)^\wap$ is the canonical inclusion.
By continuity, it follows that $\theta_1 \circ \theta^\wap$ is the identity on
$(G_d)^\wap$, and so $\theta^\wap$ must be a homeomorphism onto its range.
\end{proof}

\begin{lemma}\label{when_k0_trivial}
The following are equivalent:
\begin{enumerate}
\item\label{wkt:one} $G$ is compact;
\item\label{wkt:two} $K_0$ is empty;
\item\label{wkt:three} $\theta^\wap$ maps into $K_\wap \setminus K_0$.
\end{enumerate}
\end{lemma}
\begin{proof}
As $K_0$ is the inverse image of $\{\infty\}$ under $\iota^\infty_*$, it
is immediate that if $G$ is compact, then $K_0$ is empty.  So (\ref{wkt:one})
implies (\ref{wkt:two}), and clearly (\ref{wkt:two}) implies (\ref{wkt:three}).

Suppose that $G$ is not compact.  Then we can find some net $(s_\alpha)\subseteq G$
which eventually leaves every compact subset of $G$.  But moving to a subnet if necessary,
we may suppose that $(s_\alpha)$ converges in $(G_d)^\wap$, to $s$ say.
Notice that in $K_\wap$, we have $\theta^\wap(s) = \lim_\alpha \theta^\wap(s_\alpha)
= \lim_\alpha \theta(s_\alpha)$.  As $\iota^*_\infty:K_\wap\rightarrow G_\infty$ is
continuous, it follows that
\[ \iota^*_\infty \theta^\wap(s) = \lim_\alpha \iota^*_\infty \theta(s_\alpha)
= \lim_\alpha s_\alpha = \infty. \]
Hence $\theta^\wap(s) \in K_0$, and so we have shown that (\ref{wkt:three})
implies (\ref{wkt:one}).
\end{proof}

As $(G_d)^\wap$ is not a group, we cannot apply Lemma~\ref{semigp_lemma}.
However, in \cite{chou}, it is shown that unless $G$ is finite,
$\wap(G_d) / c_0(G_d)$ contains a copy of $\ell^\infty$.  In particular,
$\wap(G_d)$ is ``large'', and so also $K_\wap$ is ``large'' in this sense.
The following shows, again informally, that $K_\wap\setminus K_0$ is also ``large''.

\begin{proposition}
Let $G$ be non-discrete.  For any compact, non-discrete subset $A\subseteq G$,
the image of $A$ in $(G_d)^\wap$ is not closed.  However, the image of
the closure of $A$, under $\theta^\wap$, is contained in $K_\wap\setminus K_0$.
\end{proposition}
\begin{proof}
The inclusion $G_d\rightarrow G^\wap$ is continuous, so by the universal property,
we get a continuous map $\phi:(G_d)^\wap \rightarrow G^\wap$ which has dense range.
As $(G_d)^\wap$ is compact, it follows that $\phi$ is surjective.
We can see $\phi$ in a more concrete way.  By \cite[Section~4.2]{BJM},
$\wap(G) = C(G)\cap\wap(G_d)$.  By considering both $\wap(G)$ and $\wap(G_d)$
as subalgebras of $\ell^\infty(G)$, we see that the inclusion map $\wap(G)
\rightarrow\wap(G_d)$ is a $*$-homomorphism, and so induces a continuous map
$\phi:(G_d)^\wap \rightarrow G^\wap$.

As also $C_0(G) \subseteq \wap(G)$, the above argument (compare with
Proposition~\ref{iota_homo}) shows also the existence of a continuous homomorphism
$\psi:(G_d)^\wap \rightarrow G_\infty$ such that $\psi(s) = s$ for each $s\in G$.

Suppose that $A\subseteq G$ is compact and that the image of $A$ in $(G_d)^\wap$,
say denoted by $A_0$, is closed.  We can hence consider the restriction
$\psi|_{A_0} : A_0 \rightarrow G_\infty$.  Then
$\psi|_{A_0}(s) = s$ for each $s\in A$, and so $\psi|_{A_0} : A_0 \rightarrow A$
is a continuous bijection between compact sets, and is hence a homeomorphism.

We then claim that for each $f\in\wap(G_d)$, there exists $F\in C_0(G)$ such
that, if $C_0(G)$ is considered as a subspace of $\ell^\infty(G)$, then
$f(s)=F(s)$ for each $s\in A$.  By the Tietze extension theorem, we can simply 
let $F$ be an extension of the map $f \circ \psi|_{A_0}^{-1} : A\rightarrow
\mathbb C$.  Thus $f$ is continuous on $A$.
However, $c_0(G) \subseteq \wap(G_d)$, so we have shown that the restriction
of any $c_0(G)$ function to $A$ is continuous.  This implies that $A$ must be
discrete, as required.

Finally, let $A\subseteq G$ be compact, let $(s_\alpha)$ be a net in $A$, and
suppose that $s_\alpha\rightarrow s$ in $(G_d)^\wap$.  This means that
$f(s_\alpha)\rightarrow f(s)$ for each $f\in\wap(G_d)$, hence for all $f\in C_0(G)
\subseteq \wap(G) \supseteq \wap(G_d)$.  So $(s_\alpha)$ must converge in $G$,
and hence in $A$, say to $t\not=\infty$.  Then, as in the previous lemma,
$\iota^*_\infty \theta^\wap(s) = t$, so that $\theta^\wap(s) \not\in K_0$, as required.
\end{proof}

Exactly the same proof as used in Proposition~\ref{kzero_non_triv} shows that
$K_\wap \not= (G_d)^\wap$ when $G$ is non-discrete.  We finish by mentioning that,
suitably modified, the results of Section~\ref{antipode} apply to the $\wap$ case,
although this seems to give little insight, given, again, that $(G_d)^\wap$ is not
a group.  Similarly, it seems natural to ask about invariant measure on $K_\wap$,
but we have made no progress in this direction.

\bigskip
\noindent\textbf{Author's address:}
\parbox[t]{5in}{School of Mathematics,\\
University of Leeds,\\
Leeds LS2 9JT\\
United Kingdom}

\smallskip
\noindent\textbf{Email:} \texttt{matt.daws@cantab.net}

\end{document}